%% file: main_lcss_2022_revision.tex
\documentclass[letterpaper, 10 pt, conference]{ieeeconf}

\IEEEoverridecommandlockouts                              
\usepackage{changes}
\usepackage{cite}
\usepackage{amsmath,amssymb,amsfonts}
\usepackage{algorithm}
\usepackage{comment}
\usepackage{algpseudocode} 
\usepackage{tabularx}
\algtext*{EndWhile}
\algtext*{EndIf}
\algtext*{EndFor}

\usepackage{graphicx}
\usepackage{textcomp}
\usepackage{xcolor}

\usepackage{paralist} 
\usepackage{todonotes}

\usepackage{tikz}
\usetikzlibrary{shapes,arrows,fit,backgrounds,automata}
\usetikzlibrary{shapes.geometric,fit}
\usepackage[utf8]{inputenc}

\tikzset{main node/.style={circle,draw,minimum size=1cm,inner sep=0pt},}
\tikzstyle{container} = [draw, rectangle, inner sep=0.1cm]
\tikzstyle{container-ellipse} = [draw, ellipse, inner sep=0.1cm]

\usepackage{acronym}
\usepackage{amsthm}

\usepackage{soul}

\newcommand{\jie}[1]{\textcolor{blue}{JF: #1}}
\definecolor{darkgreen}{rgb}{0,0.5,0}

\makeatletter
\newcommand*{\da@rightarrow}{\mathchar"0\hexnumber@\symAMSa 4B }
\newcommand*{\da@leftarrow}{\mathchar"0\hexnumber@\symAMSa 4C }
\newcommand*{\xdashrightarrow}[2][]{%
  \mathrel{%
    \mathpalette{\da@xarrow{#1}{#2}{}\da@rightarrow{\,}{}}{}%
  }%
}
\newcommand{\xdashleftarrow}[2][]{%
  \mathrel{%
    \mathpalette{\da@xarrow{#1}{#2}\da@leftarrow{}{}{\,}}{}%
  }%
}
\newcommand*{\da@xarrow}[7]{%
  \sbox0{$\ifx#7\scriptstyle\scriptscriptstyle\else\scriptstyle\fi#5#1#6\m@th$}%
  \sbox2{$\ifx#7\scriptstyle\scriptscriptstyle\else\scriptstyle\fi#5#2#6\m@th$}%
  \sbox4{$#7\dabar@\m@th$}%
  \dimen@=\wd0 %
  \ifdim\wd2 >\dimen@
    \dimen@=\wd2 %
  \fi
  \count@=2 %
  \def\da@bars{\dabar@\dabar@}%
  \@whiledim\count@\wd4<\dimen@\do{%
    \advance\count@\@ne
    \expandafter\def\expandafter\da@bars\expandafter{%
      \da@bars
      \dabar@ 
    }%
  }%
  \mathrel{#3}%
  \mathrel{%
    \mathop{\da@bars}\limits
    \ifx\\#1\\%
    \else
      _{\copy0}%
    \fi
    \ifx\\#2\\%
    \else
      ^{\copy2}%
    \fi
  }%
  \mathrel{#4}%
}
\makeatother

\input{defs-journal}

\title{Qualitative Planning in Imperfect Information Games with Active Sensing and Reactive Sensor Attacks}

\title{Synthesizing Attack-Aware     Control and Active   Sensing Strategies under Reactive Sensor Attacks}

\author{Sumukha Udupa, Abhishek N. Kulkarni, Shuo Han, Nandi O. Leslie, Charles A. Kamhoua, and Jie Fu$^\ast$
\thanks{S. Udupa, A. Kulkarni and J. Fu ($^\ast$ corresponding author) are with the Dept. of Electrical and Computer Engineering, University of Florida, Gainesville, Fl 32611 USA.
{\tt\small \{sudupa, ankulkarni,jfu2\}@ufl.edu}}
\thanks{S. Han is with the Department of Electrical and Computer Engineering, University of Illinois, Chicago, IL 60607.
{\tt\small hanshuo@uic.edu}}
\thanks{N. Leslie is with Raytheon Technologies. 
{\tt\small nandi.o.leslie@raytheon.com}}
\thanks{C. Kamhoua is with U.S. Army
Research Laboratory. 
{\tt\small charles.a.kamhoua.civ@mail.mil}}
}

\begin{document}

\maketitle
\thispagestyle{empty}
\pagestyle{empty}

\begin{abstract}
We consider the probabilistic planning problem for a defender (P1) who can jointly query the sensors and take  control actions to reach a set of goal states while being aware of possible sensor attacks by an adversary (P2) who has perfect observations.
To synthesize a provably-correct, attack-aware \edit{joint} control and active sensing strategy for P1, we  construct a stochastic game on graph \edit{with augmented states that} include the actual game state \edit{(known only to the attacker)}, the belief of the defender about the game state (constructed by the attacker 
\edit{based on his knowledge of defender's observations)}.
We present an algorithm to \edit{compute} a belief-based, randomized strategy for P1 to ensure satisfying the reachability  objective with probability one, under the worst-case sensor attacks carried out by an informed P2. \edit{We prove the correctness of the algorithm and illustrate it using an example}.
\end{abstract}

\section{Introduction}
In this work, we develop  a formal methods based approach \edit{to synthesize}  provably correct  attack-aware cyber-physical systems (CPSs), featured by strategic interactions between a controller/defender and an attacker who carries out sensor attacks to the system. We address the following question: Given the objective of reaching a subset of states in the system, how does one   plan the defender's control actions and active information acquisition in order to satisfy the objective with probability one, under the worst-case sensor attack strategy?

 \edit{
As a motivating example, consider Fig.~\ref{fig:Example1}, where a UAV must reach the flag before its battery is depleted.  When the UAV encounters a cloud, it stops moving forward until the cloud moves away. The cloud moves randomly.
To complete the task, the UAV deploys a network of sensors to detect the cloud's location. 
An adversary may attack the sensors to prevent the UAV from accomplishing its mission.
Examples of such adversarial interactions include security patrolling robots or search and rescue in a contested environment.}

\begin{figure}
    \centering
    \includegraphics[scale=0.20]{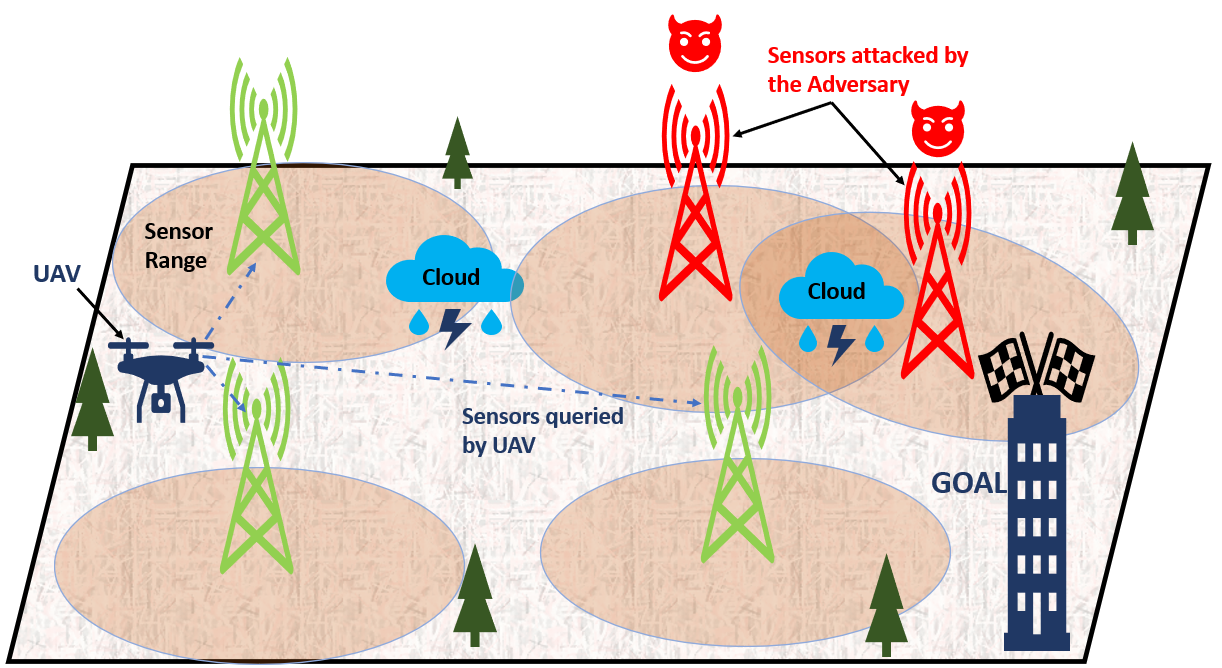}
    \caption{The strategic interaction between a controller and an adversary.}
    \vspace{-2em}
    \label{fig:Example1}
\end{figure}

 We model the interaction between the defender and the attacker as a partially observable \emph{stochastic} system   where the defender's observation  is partially controlled by the attacker: At each time step, the defender can choose what sensors to query and control actions to take, and the attacker can choose what sensors to attack. The defender receives observations from the uncompromised sensors and aims to reach a set of \emph{goal states}. 
In our previous work \cite{kulkarniQualitativePlanningImperfect2021a}, 
we analyzed  the cost of attack-unaware control where the defender mistakes the compromised sensors as probabilistic sensor failures.  
This work investigates  the synthesis of \emph{attack-aware controllers with active perception and control}. The key insight is that by knowing which sensors are susceptible to attacks, the defender can selectively choose which sensors to query \emph{in anticipation of the attacker's best response}. Our solution assumes  the worst-case scenario of asymmetric information: The attacker observes  the  state and the defender's action before  deciding  which  sensors  to  attack. Under such a worst-case attacker's information, the  attack-aware controller, if \edit{it exists},   can  provide  a  strong  guarantee  of  the  correctness of  the closed-loop system.

\edit{
\textbf{Our Approach and Contributions. } First, we model the adversarial sensor attacks with a new class of turn-based, one-sided partially observable stochastic games (POSGs), in which the observation function is dynamically changing and jointly determined by the defender and the attacker. 
Second, we construct an augmented game  in which  a state includes  the actual game state (known by the attacker), the belief of the defender  about the game state (constructed by the attacker \edit{who knows the defender's observation and the actual state)}. We develop an algorithm to solve a belief-based \ac{asw} strategy for the defender in the augmented game and prove that this strategy ensures that the control objective  in the original system can be satisfied with probability one, regardless of any sensor attack strategy. The problem is EXPTIME-complete, which matches the lower-bound complexity for one-sided POSG with a fixed observation function \cite{chatterjeePartialObservationStochasticGames2012} \edit{(see also the survey on stochastic games on graphs \cite{chatterjee2012survey})}. 
}
%


\textbf{Related Work. }
Our work is closely related to  supervisory control of discrete event systems under sensor/actuator attacks. In the literature of supervisory control,
the authors~\cite{wangSupervisoryControlDiscrete2019} studied  various sensor and actuator attacks, including  replacement, injection, deletion and replay of  observable and controllable events  to a discrete event system (DES) as a transducer and investigated the controllability of the system and the design of attack-resilient supervisory controllers. 
Sensor deception  attacks have been studied in 
\cite{meira-goesSynthesisSensorDeception2019,meira2020synthesis}  
 from the attacker's perspective. The goal is to synthesize a sensor deception attack strategy which misleads the system to reach unsafe states. Given that the system is modeled as a probabilistic automata, the authors in \cite{meira-goesSynthesisSensorDeception2019} proposed to construct a $1\frac{1}{2}$ stochastic graph game, also known as an \ac{mdp},  and employ linear program solution of \ac{mdp}s to design the optimal   strategy that maximizes the attack success probability.  Covert attacks are investigated in \cite{linSynthesisCovertActuator2020} for DES and \cite{taiSynthesisCovertSensor2021} for networked DES, where the attacker's goal is to remain hidden while compromising the system via stealthy sensor/actuator attacks. 

\edit{Our game model is different from both deterministic and stochastic DES models. First, deterministic DESs capture the adversarial interactions by a deterministic transition system with controllable and uncontrollable events and observable and unobservable events. In our model, the system dynamics is stochastic and the observation is partially controlled due to the presence of sensor attacks. In works on stochastic DESs \cite{Ratnesh2017,meira-goesSynthesisSensorDeception2019}, the system is modeled as a probabilistic automaton that specifies, for each state, the probability distribution over possible events. This model also differs from our game in which the defender decides an action, and the outcome of that action is determined by a probability distribution. Our model can reduce to a stochastic DES if the defender's policy is fixed. However, we aim to synthesize an attack-aware control strategy for the defender and thereby use the two-player stochastic game formulation. }

\section{Preliminaries and Problem Formulation}
\label{sec:preliminaries_and_problem_formulation}

\edit{Given a finite set $X$, $\dist(X)$ denotes the set of all probability distributions over $X$, and for a distribution $d\in \dist (X)$, $\supp(d) = \{x\in X\mid d(x)>0\} $ denotes the support of $d$.} 


We introduce a new class of partially observable stochastic games played under asymmetric information. In this game, an autonomous agent (Player 1, P1) actively queries sensors to obtain task-relevant information. Meanwhile, an attacker (Player 2, P2) seeks to compromise P1's mission by reactively blocking the sensor information requested by P1. 



\begin{definition}[Zero-sum Stochastic Reachability Game with Partially Controllable Observation Function
]
	\label{def:system}
	A two-player stochastic game with partially controllable observation function in which P1 has a reachability objective is a tuple 
	\[
	G = \langle S, \ctrlAct, P, s_0, \Gamma, \senseAct \times \attackAct, \calO, \obs, o_0, F \rangle, 
	\]
	where 
	\begin{inparaenum}[1)]
		\item $\langle S, \ctrlAct, P, s_0 \rangle$ is an \ac{mdp} where $S$ is a finite set of states; $A$ is a finite set of actions; $s_0$ is an initial state; $P: S \times A \rightarrow\dist(S)$ is a probabilistic transition  function such that
	$P(s,a,s')$ is the probability of reaching state $s'$ given action $a$ taken at state $s$.	
%
		\item  $\Gamma= \{0,1,\ldots, N\}$ is a set of indexed sensors. 
%
		\item $\senseAct\subseteq 2^\Gamma$ is a set of sensor query actions of P1, each of which acquires sensing information from a   subset of sensors from $\Gamma$;
		\item $\attackAct \subseteq 2^\Gamma$ is a set of sensor attack actions of P2, each of which blocks sensing information of a subset of sensors from $\Gamma$, similar to jamming attacks \cite{tayebiWirelessSensorNetwork2015,groverJammingAntijammingTechniques2014}. 
		\item $\calO \subseteq 2^S$ is a   set of observations. 
		\item $\obs: S \times \senseAct \times \attackAct \rightarrow \calO$ is a deterministic observation function \edit{of P1}, which maps a state $s \in S$, a sensor query action $\sense$, and a sensor attack action $\att$ into an observation $o = \obs(s, \sense, \att) \in \calO$. 
		\item $o_0 \in \calO$ is an initial observation and $s_0\in o_0$. 
		\item $F \subseteq S$ is a set of final states. \edit{P1 must enforce a visit to $F$ to satisfy such a \emph{reachability objective}}.
	\end{inparaenum}
\end{definition}
\edit{In contrast to the common POSG models \cite{raskin2007algorithms} where the observation functions are fixed, in our game, the observation function of P1 is determined dynamically by P1's active sensing and P2's reactive sensor attacks. }
    \edit{In particular, the observation generated due to P1's sensor query and P2's sensor attack is understood as follows: Each sensor $i \in \Gamma$ covers a subset $S_i$ of states $S$. Assuming $s$ to be the current state, sensor $i$ returns a Boolean value \edit{$v_i$: $v_i=\mathsf{True}$} if $s\in S_i$, otherwise $v_i= \mathsf{False}$. Given a state $s\in S$, a sensing action $\sense \in \senseAct$ and a sensor attack action $\att \in \attackAct$, the observation 
     $\obs(s, \sigma, \beta)$ of state $s$ is given by }
    %
    %
    \edit{
    \[
    \obs(s, \sigma, \beta)  = S - \left(\bigcup_{i\in \sigma\setminus \beta, v_i = \mathsf{True}} S\setminus S_i  \cup \bigcup_{i \in \sigma\setminus \beta, v_i=\mathsf{False}} S_i\right)
    \]}
    %
%
%
%

Two states $s, s' \in S$ are said to be observation equivalent given the sensing action and sensor attack action $\sense,\att$ if $\obs(s,\sense,\att)= \obs(s',\sense,\att)$. 
%
%

\edit{
\textbf{Information structure. }
In this game, we assume that 1) P2 has perfect observation of states and actions, \ie, P2 can directly observe the current state and P1's control and sensing actions. 2) P1 knows which sensors are attacked by P2---this assumption holds for jamming attacks.}

\textbf{Game Play.} The game play in $G$ is constructed as follows. From the initial state $s_0$, P1 obtains the initial observation $o_0$. Based on the observation, P1 selects a control action $a_0 \in \ctrlAct$ and a sensor query action $\sense_0 \in \senseAct$. The system moves to state $s_1$ with probability $P(s_0, a_0, s_1)$. At state $s_1$, P2 selects an attack action $\att_0 \in \attackAct$. The system generates a new observation $o_1 = \obs(s_1,\sense_0,\att_0)$ determined by the state, P1's sensing action and P2's sensor attack action. We denote the resulting play as $\rho = s_0 (a_0, \sense_0)  s_1  \att_0 (a_1, \sense_1)  s_2 \att_1 \ldots$. Note that P2's attack action is taken after P2 \edit{observes} the current state. 
The set of plays in $G$ is denoted by $\plays(G)$ and the set of finite prefixes of plays is denoted by $\prefplays(G)$.

\textbf{Observation Equivalent Plays to P1.} Given a play $\rho   = s_0 (a_0, \sense_0)  s_1  \att_0 (a_1, \sense_1)  s_2 \att_1 \ldots$,  P1's \emph{observation} of  $\rho$ is $ \rho^o= o_0 (a_0, \sense_0,\att_0) o_1 (a_1, \sense_1, \att_1)\ldots $ where $o_{i+1}=\obs(s_{i+1}, \sense_i, \att_i)$ for all $i\ge 0$ and $o_0$ is the initial observation. For notation convenience, we denote the observation of play $\rho$ as $\obs(\rho)$.
Two plays (or play prefixes) $\rho, \rho'$ are said to be observation-equivalent to P1, denoted by $\rho \sim \rho'$, if and only if $\obs(\rho)=\obs(\rho')$.

\textbf{P1's Reachability Objective and Strategy.} A play $\rho =s_0 (a_0 ,\sense_0) s_1 \att_0  (a_1, \sense_1)s_2 \att_1 \ldots $ is \emph{winning} for P1 if $s_k\in F$ for some $k\ge 0$. Otherwise, it is winning for P2.  During interaction, P1 must determine, simultaneously, a control action $a\in A$, and a sensor query action $\sense \in \senseAct$. We denote P1's set of actions by $\act_1  = A \times \senseAct$, and that of P2 by $\act_2= \attackAct$.
A finite-memory, randomized (resp., deterministic) strategy for player $j \in \{1, 2\}$ is a function $\pi_j : \prefplays(G) \rightarrow \dist(\act_j)$ (resp., $\pi_j : \prefplays(G) \rightarrow \act_j$). A player $j$ is said to follow strategy $\pi_j$ if for any prefix $\rho\in \prefplays(G)$ at which $\pi_j$ is defined, player $j$ takes the action $\pi_j(\rho)$ if $\pi_j$ is deterministic, or an action $a \in \supp(\pi_j(\rho))$ with probability  $\pi_j(\rho,a)$ if $\pi_j$ is randomized.
A strategy is said to be \emph{observation-based} if $\pi_j(\rho) = \pi_j(\rho')$ whenever $\rho \sim \rho'$. 

\begin{problem}
	\label{prob:main}
\edit{Given a game $G$ in Def.~\ref{def:system},}
determine if there exists an observation-based strategy 
using which   P1 can satisfy the reachability objective with probability one, for any sensor attack strategy played by P2 with perfect observations.
\end{problem}
\section{Synthesizing attack-aware strategies with active perception}
\label{sec:Synthesis}Given P2's perfect observation,  P2 can construct P1's belief given his information and higher-order information (what P2 knows P1 observes). 
  To solve Problem~\ref{prob:main}, 
  \edit{we reduce} the original game in Def.~\ref{def:system} into a two-player stochastic game whose states include \emph{P2's belief of P1's belief.}


To ease the construction, we introduce a function \edit{$\post_G: 2^S \times A \rightarrow 2^S$ that maps a given state-action pair $(s,a)$ to the possible reachable states. We then denote $\post_G(\{s\},a)$ as $\post_G(s,a)$. We also define $\post_G(B,a) = \{s' \in S \mid \exists s \in B: P(s,a,s')>0 \}$ where $B \subseteq S$ is a subset of states.}

\begin{definition}
    \label{def:two_player_game}
	Given the zero-sum stochastic game with partially controllable observations $G$ (Def.~\ref{def:system}), the stochastic two-player reachability game augmented with P2's belief and P2's belief of P1's belief is a tuple 
	\[
	\mathcal{G} = \langle Q \cup \{q_F\}, \act_1\cup\act_2, \delta, q_0, q_F \rangle,
	\]
	where \begin{itemize} 

		\item $Q = Q_1\cup Q_N\cup Q_2$ is the state set \edit{consisting of P1, P2 and nature states (c.f. \cite{chatterjee2012survey})}. $Q_1= \{(s, B) \mid s\in S, B\subseteq S\}$ is the set of states where P1 selects a (control and sensing) action $(a, \sense)$. $Q_N=\{(s,B,a,\sense)\mid s\in S, B\in 2^S,  (a,\sense)\in \act_1\}$ is the set of nature's states. $Q_2= \{(s,B, \sense)\mid s \in S, B \in 2^S, \sense \in \Sigma\}$ is a set of states where P2 selects a sensor attack action.
%
		\item $q_F$ is a single final state.  It is also a sink state. 
		\item $\act_1 =A\times \Sigma$ is a set of P1's actions and $\act_2= \attackAct$ is a set of P2's actions. 
		\item $q_0 = (s_0,o_0)$ is the initial state. 
		\item \edit{$\delta: (Q_1 \times \act_1) \cup Q_N \cup (Q_2 \times \act_2) \rightarrow \dist(Q \cup \{q_F\})$ is the probabilistic transition function  defined as follows:}
       \begin{inparaenum}[]
			\item For a P1's  state $(s, B) \in Q_1$ and action $(a, \sense) \in \act_1$,   $\delta((s,B),(a,\sense), (s,B',a,\sense)) =1$,  where $B' = \post_G(B,a)$. That is, with probability one, a nature's state $(s,B',a,\sense )$ is reached. 
			\item For a nature's state $(s,B',a,\sense) \in Q_N$, we distinguish three cases:
			\begin{inparaenum}[1)]
    			\item If $\post_G(s,a)\subseteq F$ then $\delta((s, B',a,\sense),q_F)=1$. 
    			\item If 	$\post_G(s,a)\cap F =\emptyset$, then   $\delta((s,B',a,\sense), (s',B',\sense)) =P(s,a,s')$. 
    			\item If $\post_G(s,a)\cap F\ne \emptyset$ and $\post_G(s,a) \setminus F \ne\emptyset$, then, for some $\epsilon \in (0,1)$, $\delta((s,B',a,\sense),q_F) =\epsilon$ and $\delta((s,B',a,\sense), (s',B',\sense)) =1-\epsilon$.     
    			 That is, with some positive probability $\epsilon$, the final, sink state $q_F$ is reached. 
			 \end{inparaenum}
			 %
			\item For a P2's state $(s',B',\sense) \in Q_2$, and an attack action $\att \in \act_2$,  $\delta((s',B',\sense),\att,(s', B'')) = 1$ where $B'' =B'\cap \obs(s', \sense,\att)$.  
        \end{inparaenum}
\end{itemize}
\end{definition}

A sequence of  transitions $(s,B)\xrightarrow{(a,\sense)} (s, B', a,\sense) \dashrightarrow (s',B', \sense)\xrightarrow{\att} (s',B'') $ is understood as follows: At the state $(s,B)$, the true state is $s$ and P1 believes any state in $B$ is possibly the true state. P1 selects a pair of control and sensing actions $(a,\sense)$ and updates $B$ to $B'$, which includes a set of states that may be reached if action $a$ is taken at some state in $B$. Then,  the nature player makes a probabilistic transition (represented by the dash arrow) to a new state $s'$ according to the distribution $P(s,a,s')$. P2 observes the true state $s'$ and, \edit{then, chooses} a sensor attack action $\att$. With this sensor attack, P1 observes $\obs(s', \sigma, \att)$ and updates P1's belief to eliminate states that are not consistent with the observation.
 
\edit{Def.~\ref{def:two_player_game} makes it explicit that while P2 cannot directly control the true state of the game, P2 can affect the augmented state of   game $\cal G$  indirectly by influencing the belief of P1.} \edit{ 
Our belief structure is inspired by stochastic games with signals \cite{bertrand2017qualitative}, where a player constructs a belief of his own and the belief of his opponent's belief.  However, \cite{bertrand2017qualitative} models the player interactions as a stochastic game with signals. Our modeling and solutions are different from \cite{bertrand2017qualitative}. } 

Next, we describe how to use the game $\mathcal{G}$ augmented with beliefs  to solve an attack-aware strategy in the original game $G$. \edit{First, we show that when P2 is limited to blocking sensor readings, regardless of P2's attack, P1 is sure that one of the state in P1's belief is the true state.} 
 
\begin{lemma}
\label{lma:containment}
 If a state $(s,B)$ is reachable from the initial state $q_0$, then $s\in B$.
	\end{lemma}
\begin{proof}
	By induction. The initial state $q_0 =(s_0,o_0)$ satisfies the condition (See Def.~\ref{def:system}). Consider a play in the game $\mathcal{G}$ such that $q^{1}_k$ is the $k$-th state reached by a sequence of players' actions (P1, P2's actions and the nature's stochastic choices). Suppose $q^1_k= (s_k,B_k)$ that satisfies $s_k \in B_k$. For any action  $(a,\sense)\in \act_1$ of P1,  the next state reached is $(s_k, \post_G(B_k,a), a, \sense)$. 
 	From that state, the nature's probabilistic action will determine the next state $(s_{k+1},\post_G(B_k,a), \sense)$. Note that because $s_k\in B_k$, then $s_{k+1}\in \post_G(B_k,a)$ by construction.
	
	Then, the attacker P2 takes an action $\att$ to generate an observation for P1,  $o=\obs(s_{k+1}, \sense, \att)$. As the attacker can only hide sensor readings, it holds that $ \obs(s_{k+1}, \sense, \lambda) \subseteq \obs(s_{k+1}, \sense, \att) $
	where $\lambda$ means no attack.
	And $s_{k+1}\in  \obs(s_{k+1}, \sense, \lambda)$ implies $s_{k+1}\in  \obs(s_{k+1}, \sense, \att)$. The new belief for P1 is $B_{k+1}= o \cap \post_G(B_k,a)$ and since $s_{k+1}\in o$ and $s_{k+1}\in \post_G(B_k,a)$,   it holds that $s_{k+1}\in B_{k+1}$.
	\end{proof}


 \edit{This property is critical to construct P1's observation-based strategy to reach $F$, even if P1 may not know when $F$ is reached. Consider a transition $(s,B', a,\sense)\dashrightarrow q_F $ where $\post_G(s,a)\cap F \ne \emptyset$. Since $\post_G(s,a) \subseteq B'$ implies $B'\cap F\ne \emptyset$,  P1 \emph{knows}, without observing, the probability that  $F$ is reached is  greater than $0$.} 

\begin{definition}[Belief-based Almost-Sure Winning Strategy/Region]
Given the two-player game $\mathcal{G}$, 
a strategy $\pi_1$ is \emph{almost-sure winning} for P1 if by following $\pi_1$, regardless of P2's strategy, P1   ensures to reach $q_F$ with probability one.
A strategy $\pi_1$ is \emph{belief-based} provided that for two states $(s,B), (s',B') \in Q_1$,  if $B=B'$ then $\pi_1((s,B)) = \pi_1((s',B'))$. 
A set of states from which P1 has a \emph{belief-based, almost-sure winning strategy} is called P1's almost-sure winning region with partial observation, denoted $\asw_1$.
\end{definition}
Note that any belief-based strategy is  observation-based  because the belief is constructed from   P1's observations. \edit{Next, we prove that by solving the game $\mathcal{G}$ in Def. 2, we can obtain a joint control and active sensing strategy to satisfy the objective against sensor attacks in the game $G$.}

\begin{theorem}	\label{thm:belief-based-asw}
	\edit{A belief-based almost-sure winning strategy to reach $\{q_F\}$ in P1's belief-based game $\calG$ is also almost-surely winning for P1 to visit $F$ in the game with partially controllable observation function game, $G$, regardless of the sensor attack strategy carried out by P2. }
\end{theorem}
 \begin{proof}
 By the construction of the game $\mathcal{G}$, the event of reaching $q_F$ is conditioned on the event that  a nature state $(s,B,a,\sense) \in Q_N$ where $\post_G(s,a)\cap F \ne \emptyset$ or $\post_G(s,a)\subseteq F$ is visited.  
 Let $Y \subseteq Q_N$ be all nature states that can be reached prior to visiting $q_F$ given the almost-sure winning strategy \edit{$\pi$}. 
 If $q_F$ is visited with probability one from any state in the almost-sure winning region, then the set $Y$ must be visited with probability one from any state in $\asw_1$. Let $p = \min_{(s,B, a,\sense)\in Y}\Pr(F\mid s,a)$ be the minimal probability of reaching $F$ from a state in $Y$. The probability of not reaching $F$ in $k$  visits to $Y$ is smaller than $(1-p)^k$. In addition, if $F$ is not reached, the almost-sure winning strategy will reach some state $q' \in \asw_1$ from which $Y$ is visited again with probability one.
 Hence, the probability of eventually reaching $F$ is $\lim_{k\rightarrow \infty}1- (1-p)^k=1$. That is, $\pi$ is also almost-surely winning to visit $F$ in game $G$.
%
%
\end{proof}

Next, we introduce Alg.~\ref{alg:posg-reachability} to compute a belief-based, \ac{asw} \emph{randomized} strategy for P1. The algorithm includes the following steps: 
 In the first step,  we use the solution of two-player stochastic games with two-sided perfect observations \cite{bloemGraphGamesReactive2018},   to solve the positive winning region for P2, denoted $\poswin_2 \subsetneq Q$,  which includes a set of states from which P2 can ensure a winning play with a positive probability, when both players have perfect observations. Starting from any state $q\in \poswin_2$, if P1 cannot   reach $q_F$ with probability one even if P1 has perfect observation, then P1 cannot   reach $q_F$ with probability one given partial observations.
 
 In the second step, we   initialize a set $Y_0 = Q\setminus \poswin_2$ and iteratively refine the set $Y_i$ to obtain $Y_{i+1}$, for $i\ge 0$. At iteration $i$,  Alg.~\ref{alg:posg-reachability} computes a set of states, from which P1 can ensure to stay within $Y_i$ with probability one, and with a positive probability, to reach $q_F$ in finite steps.
 The following functions are defined: 
 For each P1's state $q\in Q_1$, $Y \subseteq Q$, let 
\[
\allow(q, Y)=\{(a,\sense) \in \act_1 \mid \post_{\cal G}(q,(a,\sense)) \subseteq Y\},
\]
where $\post_{\cal G} (q, (a,\sense)) =\{q' \mid \delta (q, (a,\sense),q') > 0 \}$ is the set of states that can be reached given P1 applies $(a,\sense)$ at state $q$. By definition, P1 ensures that the next state stays within $Y$ by  taking an allowed action in $\allow(q, Y)$.

Given $q=(s,B)\in Q_1$, let $[q]_\sim = [(s,B)]_\sim = \{(s', B)\} \mid B'=B\}$ be the set of states in which P1 has the same belief as $q$. We define 
\[
\allow([q]_\sim, Y)=\bigcap_{q' \in [q]_\sim} \allow(q',Y).
\]
That is, an action is allowed at $q$ if and only if it is allowed 
at any other state $q'$ that shares the same belief as $q$.

 Given a set $Y$ and a set $R\subseteq Y$, 
we   define three   functions:
\begin{multline*}
    \prog_1(R,Y) = \{q \mid \exists (a,\sense) \in \allow([q]_\sim,Y), \\ \mathsf{Post}_{\cal G}(q,(a,\sense))\in R\}.
\end{multline*}
which \edit{outputs} a set of states from which P1 has an allowed action to reach $R$ in one step. 
\[
\prog_2(R,Y) = \{q \mid \forall \att \in \attackAct, \mathsf{Post}_{\cal G}(q,\beta) \subseteq R\}.
\]
which \edit{outputs} a set of states from which P2 cannot prevent  reaching $R$ in the next step. 
\[
\prog_N(R,Y) = \{q \mid \supp(\delta(q))\cap R \ne \emptyset \land \supp(\delta(q))\subseteq Y\}.
\]
which \edit{outputs} a set of states from which the nature player can ensure to reach $R$ with a positive probability, while staying within $Y$ with probability one.

In the inner loop of  Alg.\ref{alg:posg-reachability}, given a set $Y$, after initializing  $R_0 = \{q_F\}$, Alg.~\ref{alg:posg-reachability} iteratively computes $R_{k+1}$ given $R_k$ for all $k> 0$ until a fixed point is reached. At iteration $k+1$, $R_{k+1}$ is obtained as the union of $R_k$, $\mathsf{Prog_1}(R_k,Y)$, $\mathsf{Prog_2}(R_k,Y)$ and $\mathsf{Prog_N}(R_k,Y)$. 
By definition, from any state in $R_{k+1}\setminus R_k$, P1 can ensure to reach $R_k$ with a positive probability. 
The iteration terminates when $R_{k+1}=R_k$. Then, this new fixed point is the new set $Y$ as the outer fixed point computation. Alg.~\ref{alg:posg-reachability} terminates when $Y_{j+1}= Y_j$ and the fixed point is P1's \ac{asw} region $\asw_1$.

\begin{algorithm}[h!]
{ 
\caption{\edit{Belief-based Almost-Sure Winning Region}}
\label{alg:posg-reachability}

\begin{algorithmic}[1]
	\item[\textbf{Inputs:}] Two-player reachability game with augmented states $\mathcal{G}$ and  P2's positive winning region $\win_2^{>0}$ in $\mathcal{G}$.
	\item[\textbf{Outputs:}] P1's \ac{asw} region $\asw_1$. 
    \edit{
    \State $j \gets 0$;~$Y_j \gets Q\setminus \poswin_2$
    \Repeat
        \State $k \gets 0$;~$R_k \gets \{q_F\}$
        \Repeat 
            \State $R_{k+1} \gets R_k \cup \prog_1(R_k, Y_j) \cup \prog_2(R_k, Y_j) \cup \prog_N(R_k, Y_j)$
            \State $k \gets k+1$
        \Until{$R_{k+1} = R_k$}
        \State $Y_{j+1} \gets R_k$;~$j \gets j+1$
    \Until{$Y_{j+1} = Y_{j}$}
    }
    \State \Return $\asw_1\gets Y_j$
    
    
    
 \end{algorithmic}
}      
\end{algorithm}


\edit{We establish the correctness and completeness of Alg.\ref{alg:posg-reachability} by showing that $\asw_1$ is indeed the \ac{asw} region of P1. And, at any state in $\asw_1$, P1 has an \ac{asw} strategy to visit  $F$.
}

 \begin{lemma}
The set $\asw_1$ obtained from Alg.\ref{alg:posg-reachability} is the almost-sure winning \edit{region} for P1.
\label{lemma:YisASW}
\end{lemma}
 \begin{proof}
 Let $N$ be the index where $Y_N=Y_{N+1} $. To prove that $Y_N= \asw_1$, we prove the following:
 1) $\asw_1 \subseteq Y_j$, for all $0\le j\le N $. By induction, first, given that  $Y_0 = Q \setminus \poswin_2$, it holds that $\asw_1 \subseteq Y_0$. 
    Assume that $\asw_1\subseteq Y_i$ for some $i> 0$, we show that $\asw_1\subseteq Y_{i+1}$ as follows: \edit{Note that $Y_{i+1} = R_{k} \cup \prog_1(R_k, Y_j) \cup \prog_2(R_k, Y_j) \cup \prog_N(R_k, Y_j)$.} By construction, $Y_{i+1}$ includes  any   state from which P1 has a strategy to reach $q_F$ with a positive probability, while staying in $Y_i$. Thus, for any state $q \in  Y_{i} \setminus Y_{i+1}$, either 1) P1 cannot ensure to stay within $Y_i$ with probability one, or 2) P2 has a strategy to ensure $q_F$ is not reached with  probability 1. For case 1),  if from a state, P1 cannot ensure to stay in $Y_i$, then that state is not in $\asw_1$. This is because for any state in $\asw_1$, P1 has a strategy to ensure staying within $\asw_1$ and thereby $Y_i$ given $\asw_1\subseteq Y_i$. A state satisfying the condition in case 2)  is P2's \ac{asw} region and thus not in $\asw_1$. Therefore, $Y_{i+1}$ only removes states that are not in $\asw_1$ from $Y_i$ and thus $\asw_1 \subseteq Y_{i+1}$.
    
   2) $Y_N \setminus \asw_1 = \emptyset$, 
   By contradiction, assume there exists a state $q \in Y_N\setminus \asw_1$. 
    \edit{By construction,  for any $q \in R_{k}\cup \prog_1(R_k, Y_j) \cup \prog_2(R_k, Y_j) \cup \prog_N(R_k, Y_j)$}, P1 has a strategy to reach $q_F$ with a positive probability in finitely many steps, regardless of the strategy of P2. Let $E_{T}$ be the event that ``starting from a state  in $Y_N$, a run reaches the final state $q_F$ in $T$ steps. '' and let $\gamma >0$ be the minimal probability for the event $E_T$ to occur for any state $q\in Y_N$.  Then, the probability of not reaching $q_F$ in infinitely many steps can be upper bounded by $\lim\limits_{k \to \infty}(1-\gamma)^k=0$. 
    Therefore, for any $q\in Y_N$,  P1 has a strategy to ensure $q_F$ is reached with probability one  \edit{and  thus ensures $F$ is reached with probability one (Theorem~\ref{thm:belief-based-asw}).  This contradicts with the assumption that $q \notin \asw_1$. }
    %
    Combining 1) and 2),  we show that $\asw_1 = Y_N$.
 \end{proof}


\edit{Given} $\asw_1$, P1's belief-based, \ac{asw} strategy is defined by a set-based function $\pi_1^\ast: \asw_1\rightarrow 2^{\act_1}$ such that
\begin{equation}
\label{eq:belief-based-strategy}
\pi_1^\ast(q) = \{(a,\sense)\mid (a,\sense)\in \allow([q]_\sim, \asw_1)\}.
\end{equation}
At each state $q\in \asw_1$, P1 must take any action in $\pi_1^\ast(q)$ with a nonzero probability. 
By definition, $\pi_1(q) = \pi_1(q')$ if $q, q'$ share the same belief.

 \begin{theorem}
    By following  $\pi_1^\ast$ defined in Eq.~\ref{eq:belief-based-strategy}, P1 ensures that the game eventually reaches the   state $q_F$. 
\end{theorem}
\begin{proof}
Let $R_0, R_1,\ldots, R_K$ be the set of level sets constructed using Alg.~\ref{alg:posg-reachability} given input $\asw_1$. For level  $0<j\le K$ and a state  $q \in R_j$, let $(a,\sense) \in \pi_1^\ast(q)$ such that $\post_{\cal G}(q,(a,\sense))\in R_{j-1}$. Because  taking the action $(a,\sense)$ has a nonzero probability,  then the level will strictly decrease with a \emph{positive} probability. In addition, with probability one, the game stays within $\asw_1$ for any action in $\pi_1^\ast(q)$ and its probabilistic outcomes.  Then, let $E_n$ be the event that ``Reaching $R_{j-1}$ from a state in  $R_j$ in $n$ steps.''  It holds that $\lim_{n\rightarrow \infty} P(E_n)  = 1$
. Thus, by \edit{repeating}  the same argument for $j=K,K-1,\ldots, 1$, eventually, $R_0=\{q_F\}$ will be reached with probability one.
\end{proof}
\begin{remark}

Note that in computation, P2's strategy is not restricted to be belief-based. Therefore, for any state $q\in \asw_1$, P1 can ensure almost-sure winning regardless of P2's strategy given P2's perfect observation. 
\end{remark}





\noindent \textbf{Complexity analysis: }
\edit{The time complexity for solving P1's \ac{asw} belief-based strategy in $\calG$ is $\calO(|Q|(|Q_1|\cdot(|A|\cdot|\Sigma|)+\abs{Q_2}\cdot|\mathcal{B}|+|Q_N|))$. 
In terms of the original game, we have the complexity to be $\calO(2^{|S|} \cdot |A| \cdot (|\Sigma| + |\mathcal{B}|))$ due to the subset construction for beliefs. The complexity matches the lower bound for one-sided partial information games  \cite{chatterjeePartialObservationStochasticGames2012}. }

\section{An illustrative example}
\label{sec:Example}


In this section, we present \edit{an example} to illustrate the proposed algorithm. Consider the \ac{mdp} shown in Fig.~\ref{fig:example_base_game},
P1 has 5 sensors, $A$, $B$, $C$, $D$, and $E$ covering the states $\{s_0,s_1\}$, $\{s_1,s_2\}$, $\{s_0,s_2,s_3\}$, $\{s_4,s_5\}$, and $\{s_2,s_6, s_7\}$ respectively and four control actions   $\{a_0,a_1,a_2,a_3\}$ with probabilistic outcomes. P1 has  four sensor query actions $\sigma_0$, $\sigma_1$, $\sigma_2$, and $\sigma_3$ which query the sensors $\{A,B\}$, $\{A,C\}$, $\{B\}$, and $\{B,E\}$ respectively. P1's goal is to reach $s_4$. 

\begin{figure}[ht!]
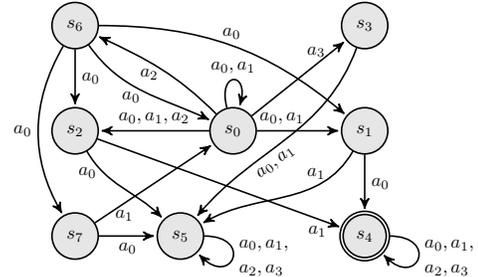

        \centering       \vspace{-1ex}
       \include{figures/example_base_graph}
        \caption{An example for attack-aware planning. We omit  exact transition probabilities and indicate the possible outcomes for each state-action pair. 
        }
      \label{fig:example_base_game}
\end{figure}

\begin{figure}[ht!]
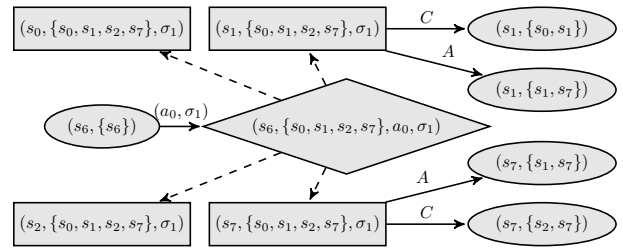

        \centering
       \include{figures/illustration_single_run}
        \caption{A fragment of the augmented game $\mathcal{G}$. P1's states are ellipses, P2's states are rectangles, and the nature player's states are diamonds.}
      \label{fig:illustration_of_run}      

 \end{figure}
\edit{Figure~\ref{fig:illustration_of_run} is a fragment of game $\mathcal{G}$ augmented with beliefs. Starting with P1's state $(s_6, \{s_6\})$, P1 takes action $a_0$ and    queries sensor $\{A,C\}$ with action $\sigma_1$. The next state is a nature state $(s_6, \{s_0,s_1,s_2,s_7\}, a_0,\sigma_1)$ where the belief is updated to $\post_G(\{s_6\},a_0)= \{s_0,s_1,s_2,s_7\} $   which are possible next states   given the action $a_0$ taken at $s_6$. Then, the nature player randomly selects one of the states, say $s_1$, and arrives at $(s_1,\{s_0,s_1,s_2,s_7\}, \sigma_1)$. Note that if there is no sensor attack, P1 should obtain $\mathsf{True}$ for sensor $A$ and $\mathsf{False}$ for sensor $C$ and deduce the current state is $s_1$. However, P2 attacks sensor $C$ so that P1 only receives reading from $A$ and deduces the current state can either be $s_0$ or $s_1$---resulting in P1's state $(s_1,\{s_0,s_1\})$. 
}

 We use three  different variations of the example to highlight the system performance given attackers with different capabilities: \textbf{Case 1:} No attack:  P2 has no sensor attack actions. In this case, P1 planning with joint control and sensing actions;  
\textbf{Case 2:} Restricted Attack:  each time P2 can attack  one of the sensors from $\{B,E\}$. \textbf{Case 3:} Unrestricted Attack: each time P2 can attack any one of the sensors. 
Assuming P1 knows the initial state, 
  under three cases, the sets of initial states from which P1 has an \ac{asw} strategy are shown in Table~\ref{tab:1}. \edit{The strategies for each of these cases were computed in $4.3$~s, $6.5$~s and $14.1$~s respectively on a laptop with AMD RYZEN $9$ processor and $16$~GB of RAM.}  
  
\begin{table}[h!]      

     \caption{Summary of   three cases}
    \label{tab:1}
 \setlength\tabcolsep{0pt}
        \begin{tabular}{|l|c|}
        \hline

Case& Almost-sure winning initial states \\
\hline
1. No Attack \& 2. Restricted Attack &$s_0$, $s_1$, $s_2$, $s_6$, $s_7$  \\ \hline


3. Unrestricted Attack &$s_0$, $s_1$, $s_2$, $s_7$ \\

\hline
    \end{tabular}
    \vspace{-3ex}
 \end{table}
  Note that starting from $s_6$, P1 has 
    \ac{asw} strategies for reaching $\{s_4\}$ in Cases 1 and 2, but not in Case 3. Consider Case 1 (no attack), from   state $s_6$, P1 only has action $a_0$ and thus, can reach one of $s_0, s_1, s_2$ or $s_7$ with some positive probability.   P1's \ac{asw} strategy assigns the following actions with nonzero probabilities: $\{(a_0,\sigma_0),(a_0,\sigma_1),(a_0,\sigma_3)\}$. Action $(a_0,\sigma_2)$ is not allowed because,   with some positive probability the next state is $s_1$ and P1 refines her belief given the sensor information to $\{s_1,s_7\}$ where P1 has no actions to ensure reaching $s_4$: In the state $s_1$, $a_0$ reaches $s_4$ and $a_1$ reaches a sink state $s_5$. However,  at $s_7$,  it is the other way around. Similar statement holds if the state $s_7$ is reached. As P1 must choose between $a_0$ and $a_1$ at $(s_1,\{s_1,s_7\})$ and $(s_7, \{s_1,s_7\})$ and there is no belief-based \ac{asw} strategy to reach $s_4$ from these two states.

Though the set of winning initial  states are the same for Case 1 and 2. The winning strategies for P1 are different. In Case 2,  P1's winning actions at $(s_6,\{s_6\})$ are $(a_0,\sigma_0)$ and $(a_0,\sigma_1)$, excluding  action $(a_0, \sigma_3)$ which was a winning action in case 1. The reason is as follows: Suppose $(s_0,\sigma_3) $ is taken, with a positive probability, the game reaches a P2's state $(s_2,\{s_0,s_1,s_2,s_7\},\sigma_3)$. P2  attacks sensor $E$ and results in  P1's state $(s_2,\{s_1,s_2\})$. Given P1's  belief  $\{s_1,s_2\}$, the action $a_0$ is winning for   state $s_1$ but is losing for   state $s_2$. Thus, P1 does not have an action at  $(s_2,\{s_1,s_2\})$ and $(s_1,\{s_1,s_2\})$ to ensure reaching $s_4$ with probability one. Thus, action $(a_0,\sigma_3)$ is not an action from P1's almost-sure winning strategy under P2's restricted attack.

Finally, in Case 3, state $s_6$ is no longer in $\asw_1$. Consider the  two actions $(a_0,\sigma_0)$ and $(a_0,\sigma_1)$ allowed by the winning strategy at $(s_6,\{s_6\})$ for Case 2. With the action $(a_0,\sigma_0)$,  the game transitions, with a positive probability,  to  P2 state $(s_1,\{s_0,s_1,s_2,s_7\},\sigma_0)$, which is not in $\asw_1$ as P2 can drive the game into the P1 state $(s_1,\{s_1,s_2\})$ by attacking   sensor $A$. Consider the action $(a_0,\sigma_1)$, with a positive probability, the game reaches P2 state $(s_7,\{s_0,s_1,s_2,s_7\},\sigma_1)$, which is not in $\asw_1$ as P2 can drive the game to reach $(s_7,\{s_1,s_7\})$ by attacking   sensor $A$. Starting from $s_6$, P1 has no strategy to reach $s_4$  if P2 can attack any sensor.


\section{Conclusion and future work}
\label{sec:Conclusion} 
In this work, we studied  qualitative planning of control and active information acquisition given adversarial sensor attacks and presented a method to synthesize an observation-based strategy, for the attack-aware defender, to satisfy a reachability objective with probability one, under the worst case sensor attacks on its observations.  With the formal-method based modeling framework and solution approach, a number of future directions are considered: 1)  Strategic sensor placement can be studied to ensure the existence of an attack-aware almost-sure winning strategy. 2) The synthesis of P1's strategies can be analyzed given other asymmetric information structures between P1 and P2, including concurrent game interactions, and two-sided partial observations. \edit{3) Symbolic approaches (c.f. \cite{chatterjeePartialObservationStochasticGames2012}) for solving POSGs may be investigated  to avoid explicit exponential subset construction.} 

\bibliographystyle{IEEEtran}
\bibliography{refs}

\end{document}

%% file: defs-journal.tex
\newcommand{\edit}[1]{#1}

\newcommand{\ie}{\textit{i.e.}}

\newcommand{\senseAct}{\Sigma}
\newcommand{\ctrlAct}{A}
\newcommand{\attackAct}{\mathcal{B}}
\newcommand{\calO}{\mathcal{O}}
\newcommand{\obs}{\mathsf{Obs}}
\newcommand{\dist}{\mathcal{D}}
\newcommand{\act}{\mathcal{A}}
\newcommand{\sense}{\sigma}
\newcommand{\att}{\beta}
\newcommand{\abs}[1]{\lvert#1\rvert}

\newcommand{\supp}{\mathsf{Supp}}
\newcommand{\plays}{\mathsf{Plays}}

\newcommand{\post}{\mathsf{Post}}
\newcommand{\prefplays}{\mathsf{Prefs}}

\newcommand{\win}{\mathsf{Win}}
\newcommand{\asw}{\mathsf{Win}^{=1}}
\newcommand{\poswin}{\mathsf{Win}^{>0}}

\newcommand{\allow}{\mathsf{Allow}}

\newcommand{\prog}{\mathsf{Prog}}

\newcommand{\calG}{\mathcal{G}}

\newtheorem{theorem}{Theorem} 
\newtheorem{definition}{Definition}

\newtheorem{problem}{Problem}

\newtheorem{remark}{Remark}

\newtheorem{lemma}{Lemma}

\acrodef{pomdp}[POMDP]{partially observable Markov decision process}
\acrodef{mdp}[MDP]{Markov decision process}
\acrodef{asw}[ASW]{Almost-Sure Winning}
\acrodef{cps}[CPS]{Cyber-Physical Systems}

%% file: figures/example_base_graph.tex
\begin{tikzpicture}[->,>=stealth',shorten >=1pt,auto,node distance=2.5cm,scale=0.7,semithick, transform shape,square/.style={rectangle}]
		\tikzstyle{every state}=[fill=black!10!white];


    \node[state] (0) at (0, 0) {$s_0$};
    \node[state] (2) at (-3, 0) {$s_2$};
    \node[state] (3) at (2.5, 2) {$s_3$};
    \node[state] (1) at (2.5, 0) {$s_1$};
    \node[state, accepting] (4) at (2.5, -2) {$s_4$};
    \node[state] (5) at (-1, -2) {$s_5$};
    \node[state] (6) at (-3, 2) {$s_6$};
    \node[state] (7) at (-3, -2) {$s_7$};

  \path (0) edge[loop above]  node{$a_0, a_1$} (0)
        (0) edge   node[pos=0.3]{$a_0, a_1$} (1)
        (0) edge   node[above]{$a_0, a_1,a_2$} (2)
        (0) edge   node[left,pos=0.85]{$a_3$} (3)
        (0) edge [out=135,in=335]  node[below,pos=0.6]{$a_2$} (6)
        (1) edge  node{$a_0$} (4)
        (1) edge [out=240,in=30] node[above,pos=0.3]{$a_1$} (5)
        (2) edge  node[below,pos=0.9]{$a_1$} (4)
        (2) edge [out=300,in=135]  node[left,pos=0.25]{$a_0$} (5)
        (3) edge  [out=250,in=50]  node[below,pos=0.6,sloped]{$a_0, a_1$} (5)
        (4) edge[loop right, out=0, in=-45,looseness=7]  node[below right]{$a_2, a_3$} node[above right]{$a_0, a_1,$} (4)
        (5) edge[loop right, out=0, in=-45,looseness=7]  node[below right]{$a_2, a_3$} node[above right]{$a_0, a_1,$} (5)
        (6) edge [out=0,in=140] node[above]{$a_0$} (1)
        (6) edge [out=305,in=155] node[below,pos=0.4]{$a_0$} (0)
        (6) edge  node{$a_0$} (2)
        (6) edge [bend right]  node[left]{$a_0$} (7)
        (7) edge   node[below,pos=0.25]{$a_1$} (0)
        (7) edge   node[below,pos=0.5]{$a_0$} (5)
        ;

\end{tikzpicture}

%% file: figures/illustration_single_run.tex
\begin{tikzpicture}[->,>=stealth',shorten >=1pt,auto,node distance=2.3cm,scale=.65,semithick, transform shape,square/.style={rectangle}]
		\tikzstyle{every state}=[fill=black!10!white];

  \node[state,ellipse] (0) 
    at (0, 0) 
    {$(s_6,\{s_6\})$};
  
  \node[state, diamond, aspect=3] (1)
    at (5, 0) 
    {$(s_6,\{s_0,s_1,s_2,s_7\},a_0,\sigma_1)$};

  \node[state, rectangle] (3) 
    at (0, 2) 
    {$(s_0,\{s_0,s_1,s_2,s_7\},\sigma_1)$};
  
  \node[state, rectangle] (2) 
    at (4, 2)
    {$(s_1,\{s_0,s_1,s_2,s_7\},\sigma_1)$};
    
  \node[state, rectangle] (4) 
    at (0, -2)
    {$(s_2,\{s_0,s_1,s_2,s_7\},\sigma_1)$};
    
  \node[state, rectangle] (5) 
    at (4, -2)
    {$(s_7,\{s_0,s_1,s_2,s_7\},\sigma_1)$};
    
  \node[state,ellipse] (6) 
    at (9, 0.75)
    {$(s_1,\{s_1,s_7\})$};
  
  \node[state,ellipse] (7) 
    at (9, 2)
    {$(s_1,\{s_0,s_1\})$};
    
  \node[state,ellipse] (8) 
    at (9, -2)
    {$(s_7,\{s_2,s_7\})$};
    
  \node[state,ellipse] (9) 
    at (9, -0.75)
    {$(s_7,\{s_1,s_7\})$};

  
  \path (0) edge node{$(a_0, \sigma_1)$} (1)
        (1) edge [dashed]  node{} (2)
        (1) edge [dashed]  node{}(3)
        (1) edge [dashed]  node{}(4)
        (1) edge [dashed]  node{} (5)
        (2) edge node{$A$} (6)
        (2) edge  node{$C$} (7)
        (5) edge node{$C$} (8)
        (5) edge node{$A$} (9)

        
        
        

  ;

\end{tikzpicture}

%% file: main_lcss_2022_revision.bbl
\begin{thebibliography}{10}
\providecommand{\url}[1]{#1}
\csname url@samestyle\endcsname
\providecommand{\newblock}{\relax}
\providecommand{\bibinfo}[2]{#2}
\providecommand{\BIBentrySTDinterwordspacing}{\spaceskip=0pt\relax}
\providecommand{\BIBentryALTinterwordstretchfactor}{4}
\providecommand{\BIBentryALTinterwordspacing}{\spaceskip=\fontdimen2\font plus
\BIBentryALTinterwordstretchfactor\fontdimen3\font minus
  \fontdimen4\font\relax}
\providecommand{\BIBforeignlanguage}[2]{{%
\expandafter\ifx\csname l@#1\endcsname\relax
\typeout{** WARNING: IEEEtran.bst: No hyphenation pattern has been}%
\typeout{** loaded for the language `#1'. Using the pattern for}%
\typeout{** the default language instead.}%
\else
\language=\csname l@#1\endcsname
\fi
#2}}
\providecommand{\BIBdecl}{\relax}
\BIBdecl

\bibitem{kulkarniQualitativePlanningImperfect2021a}
A.~N. Kulkarni, S.~Han, N.~O. Leslie, C.~A. Kamhoua, and J.~Fu, ``Qualitative
  {{Planning}} in {{Imperfect Information Games}} with {{Active Sensing}} and
  {{Reactive Sensor Attacks}}: {{Cost}} of {{Unawareness}},'' in \emph{{{IEEE
  Conference}} on {{Decision}} and {{Control}} ({{CDC}})}, 2021.

\bibitem{chatterjeePartialObservationStochasticGames2012}
K.~Chatterjee and L.~Doyen, ``Partial-{Observation} {Stochastic} {Games}: {How}
  to {Win} {When} {Belief} {Fails},'' in \emph{2012 27th {Annual} {IEEE}
  {Symposium} on {Logic} in {Computer} {Science}}, 2012.

\bibitem{chatterjee2012survey}
K.~Chatterjee and T.~A. Henzinger, ``A survey of stochastic $\omega$-regular
  games,'' \emph{Journal of Computer and System Sciences}, 2012.

\bibitem{wangSupervisoryControlDiscrete2019}
Y.~Wang and M.~Pajic, ``Supervisory {{Control}} of {{Discrete Event Systems}}
  in the {{Presence}} of {{Sensor}} and {{Actuator Attacks}},'' in
  \emph{{{IEEE}} {{Conference}} on {{Decision}} and {{Control}} ({{CDC}})},
  2019.

\bibitem{meira-goesSynthesisSensorDeception2019}
R.~Meira-Góes, R.~Kwong, and S.~Lafortune, ``Synthesis of {{Sensor Deception
  Attacks}} for {{Systems Modeled}} as {{Probabilistic Automata}},'' in
  \emph{{{American Control Conference}} ({{ACC}})}, 2019.

\bibitem{meira2020synthesis}
R.~Meira-G{\'o}es, E.~Kang, R.~H. Kwong, and S.~Lafortune, ``Synthesis of
  sensor deception attacks at the supervisory layer of cyber--physical
  systems,'' \emph{Automatica}, 2020.

\bibitem{linSynthesisCovertActuator2020}
L.~Lin and R.~Su, ``Synthesis of {{Covert Actuator}} and {{Sensor Attackers}}
  as {{Supervisor Synthesis}},'' \emph{IFAC-PapersOnLine}, 2020.

\bibitem{taiSynthesisCovertSensor2021}
\BIBentryALTinterwordspacing
R.~Tai, L.~Lin, Y.~Zhu, and R.~Su. Synthesis of {{Covert Sensor Attacks}} in
  {{Networked Discrete-Event Systems}} with {{Non-FIFO Channels}}. [Online].
  Available: \url{http://arxiv.org/abs/2103.07132}
\BIBentrySTDinterwordspacing

\bibitem{Ratnesh2017}
J.~Chen, M.~Ibrahim, and R.~Kumar, ``Quantification of secrecy in partially
  observed stochastic discrete event systems,'' \emph{IEEE Transactions on
  Automation Science and Engineering}, 2017.

\bibitem{tayebiWirelessSensorNetwork2015}
A.~Tayebi, S.~M. Berber, and A.~Swain, ``Wireless {{Sensor Network Attacks}}:
  {{An Overview}} and {{Critical Analysis}} with {{Detailed Investigation}} on
  {{Jamming Attack Effects}},'' in \emph{Sensing {{Technology}}: {{Current
  Status}} and {{Future Trends III}}}, A.~Mason, S.~C. Mukhopadhyay, and K.~P.
  Jayasundera, Eds.\hskip 1em plus 0.5em minus 0.4em\relax {Springer
  International Publishing}, 2015.

\bibitem{groverJammingAntijammingTechniques2014}
K.~Grover, A.~Lim, and Q.~Yang, ``Jamming and anti-jamming techniques in
  wireless networks: A survey,'' \emph{International Journal of Ad Hoc and
  Ubiquitous Computing}, 2014.

\bibitem{raskin2007algorithms}
J.-F. Raskin, T.~A. Henzinger, L.~Doyen, and K.~Chatterjee, ``Algorithms for
  omega-regular games with imperfect information,'' \emph{Logical Methods in
  Computer Science}, vol.~3, 2007.

\bibitem{bertrand2017qualitative}
N.~Bertrand, B.~Genest, and H.~Gimbert, ``Qualitative determinacy and
  decidability of stochastic games with signals,'' \emph{Journal of the ACM
  (JACM)}, 2017.

\bibitem{bloemGraphGamesReactive2018}
R.~Bloem, K.~Chatterjee, and B.~Jobstmann, ``Graph {{Games}} and {{Reactive
  Synthesis}},'' in \emph{Handbook of {{Model Checking}}}, E.~M. Clarke, T.~A.
  Henzinger, H.~Veith, and R.~Bloem, Eds.\hskip 1em plus 0.5em minus
  0.4em\relax {Springer International Publishing}, 2018.

\end{thebibliography}
